\documentclass[preprint,12pt]{elsarticle}

\usepackage{amssymb}
\usepackage{graphicx}
\usepackage{mathptmx}      

\usepackage{amssymb,amsfonts,amsmath,amscd,latexsym}
\usepackage[english]{babel}
\usepackage{geometry}
\usepackage{latexsym}
\usepackage{amssymb}
\usepackage[all]{xy}
\usepackage{amssymb}
\usepackage{latexsym}
\usepackage{graphicx}
\usepackage{appendix}
\usepackage[active]{srcltx}
\usepackage{srcltx}
\usepackage{amsthm}

\newdefinition{defn}{Definition}[section]
\newdefinition{nim}[defn]{}
\newdefinition{rem}[defn]{Remark}
\newdefinition{ex}[defn]{Example}
\newtheorem{thm}{Theorem}[section]
\newtheorem{cor}[thm]{Corollary}
\newtheorem{lem}[thm]{Lemma}
\newtheorem{prop}[thm]{Proposition}
\newtheorem{qu}[thm]{Question}

\newcommand{\Hom}{\operatorname{Hom}}

\newcommand{\Ext}{\operatorname{Ext}}

\newcommand{\Aut}{\operatorname{Aut}}
\newcommand{\Out}{\operatorname{Out}}
\newcommand{\Inn}{\operatorname{Inn}}

\newcommand{\HH}{\operatorname{HH}}

\newcommand{\Hc}{\operatorname{H}}

\newcommand{\Od}{\mathcal{O}}
\begin{document}

\begin{frontmatter}



\title{Nontriviality of  the first Hochschild cohomology of some block algebras of finite groups}


\author{Constantin-Cosmin Todea}
\ead{Constantin.Todea@math.utcluj.ro}
\address{Department of Mathematics, Technical University of Cluj-Napoca, Str. G. Baritiu 25,
 Cluj-Napoca 400027, Romania }

\fntext[]{This work was supported by a grant of the Ministry of Research, Innovation and Digitization, CNCS/CCCDI–UEFISCDI, project number PN-III-P1-1.1-TE-2019-0136, within PNCDI III}
\begin{abstract}
We show that for some finite group block algebras, with nontrivial defect groups, the first Hochschild cohomology is nontrivial.
Along the way we obtain methods to investigate the nontriviality of  the first Hochschild cohomology of some twisted group algebras.
\end{abstract}


\end{frontmatter}


\section{Introduction} \label{sec1-intro}
The first Hochschild cohomology space of a finite dimensional algebra over a field is isomorphic with the quotient space
of derivations of the given algebra modulo inner derivations. It is an important tool attached to an algebra, studied in many articles over the years,
which inherits a Lie algebra structure with brackets induced by the commutator derivations. One problem which is 
extensively studied in recent years, see \cite{Ch, Ei, Ru} to mention just a few,  is the problem of solvability of the Lie algebra structure of the first Hochschild cohomology.
The investigation of the first Hochschild cohomology of block algebras of finite groups and its Lie algebra structure is done in \cite{Be, Li3}. For instance in \cite{Li3} the authors show that the first Hochschild cohomology space of a finite group block algebra, having a unique class of simple modules, is a simple Lie algebra if and only if the block algebra is nilpotent 
with elementary abelian defect group of order at least $3$.
 
A result of Fleischmann, Janiszczak and Lempken in \cite{FJL} on centralizers in finite groups, 
which uses the Classification of Finite Simple Groups Theorem, implies that  the first Hochschild cohomology 
of a group algebra, $\HH^1(kG)$ is nontrivial for any finite group $G$ of order divisible by a prime $p$,
where $k$ is an algebraically closed field of characteristic $p$.
 The group algebra $kG$  decomposes into indecomposable
factors called block algebras, which correspond to primitive idempotents of the center $Z(kG)$.
Let $b$ be a block idempotent of $kG$ with $B=kGb$ its block algebra and  a
defect group $P$ of $b$, which is a certain $p$-subgroup in $G$ behaving similarly to a Sylow $p$-subgroup. Since 
$\HH^1(kG)$ decomposes into the direct product of the first Hochschild cohomology of all block
algebras of $kG$, it follows that there is a block which has nonzero first Hochschild cohomology.
 In \cite{Li} Markus Linckelmann launches the following question:
\begin{qu}\label{question}(\cite[Question 7.7]{Li}) Is it true that for any block algebra
$B$ of a finite group algebra $kG$ with $P\neq \{1\}$  we have $\HH^1(B)\neq 0$?
\end{qu}
The first cases of blocks for which Question \ref{question} has positive answer are blocks of symmetric
groups with abelian nontrivial defect groups, \cite[Example 7.5]{Li}.

Many finite group block algebras are Morita equivalent to
some twisted group algebras.  If $b$ has normal defect group in $G$  then $B$ is Morita equivalent with some twisted group algebra (with underlying group a semidirect product of the defect group and the so-called inertial quotient), see \cite{Ku1}. If $G$ is $p$-solvable, then $B$ is Morita equivalent
to some twisted group algebra over a different $p$-solvable finite group. This result is obtained by  K\"ulshammer   in \cite{Ku2}. We are now  motivated to study  the first Hochschild cohomology of twisted group algebras.

If $L$ is a finite group acting trivially on $k^{\times}$ (the group of multiplicative units in $k$)
and $\alpha\in Z^2(L,k^{\times})$ is a $2$-cocycle, 
we denote by $k_{\alpha}L$ the twisted group algebra of $L$ with respect to $\alpha$. A twisted group algebra is invariant
with respect to $2$-cocycles belonging to the same cohomology class. Since any class in $ \Hc^2(G,k^{\times})$ can be
represented by a normalized $2$-cocycle we assume in this paper that $\alpha$ is normalized ($\alpha (1,1)=1$), although we do not always need this. An element
$x\in L$ is called  $\alpha$-regular if $$\alpha(x,y)=\alpha(y,x), \forall y\in C_L(x).$$
We denote by $L_{\alpha}^{\circ}$ the set of $\alpha$-regular elements in $L$. This definition and other properties can be found in
\cite[Chapter 3, Section 6]{Ka} but, for completeness, 
Section \ref{sec2-reminder} is devoted to the presentation of some basic facts 
about $2$-cocycles and $\alpha$-regular elements which are useful
in this paper. For the rest of the paper we continue to assume that the prime $p$ is the characteristic of the field $k$. Recall that $L$ satisfies the commutator 
index property $C(p)$ if there is an element $x\in L$ such that $p$ divides $|C_L(x):(C_L(x))'|$, see \cite[Definition 1.1]{FJL}. 
In this case we will say that $x$ is an element of $L$  which gives the Commutator index property $C(p)$ for $L$.
We denote by $L_{C(p)}$ the set of all elements of $L$ which give the Commutator index property
$C(p)$ for $L$. As a consequence of \cite[Theorem]{FJL}  we obtain that $L_{C(p)}\neq \emptyset$ for any finite group $L$ such that $p$ divides the order of $L$. In the first main  result of this paper we investigate when first Hochschild cohomologies of some twisted group algebras are nontrivial.
\begin{thm}\label{thm12}
 Let $L$ be a finite group such that $p$ divides the order of $L$ and 
 $\alpha \in Z^2(L,k^{\times})$. If $L_{C(p)}\cap L_{\alpha}^{\circ}\neq \emptyset$ then $\HH^1(k_{\alpha}L)\neq 0.$
\end{thm}
 We denote by $L_{S(p)}$ the set of all $p$-elements giving the strong Non-Schur
property $S(p)$ for $L$. This is the set of $p$-elements $x\in L$ such that $x\notin C_L(x)'$. If $L_{S(p)}$ would be nonempty for all finite groups $L$, then Lemma \ref{lem24} (vi) should give  $L_{C(p)}\cap L_{\alpha}^{\circ}\neq \emptyset$. But, there are cases of finite groups for which 
$L_{S(p)}=\emptyset$, see \cite[Proposition 2.2]{FJL}.
We do not know if any finite group $L$, with its order divisible by $p$, has a $p$-element which gives
the Commutator index property $C(p)$. If this is true then any twisted group algebra of $L$ has nontrivial 
first Hochschild cohomology.   By email discussion  with Wolfgang Lempken  it seems that the exceptions of \cite[Proposition 2.2]{FJL}, the sporadic simple groups $Ru, J_4$ and  $Th$  do not have this kind of $p$-elements  for the primes  $3, 3, 5$, respectively. So, we launch the following question:

\begin{qu}\label{qutwist}Is the first Hochschild cohomology of  twisted group algebras  of  finite groups (with order divisible by $p$) nontrivial?
If not, can one find examples of such twisted group algebras which have  zero  first Hochschild cohomology?
\end{qu}
In Proposition \ref{propfgC(p)good} we collect classes of finite groups which satisfy
the assumption of Theorem \ref{thm12}. Applying Theorem \ref{thm12} and Proposition \ref{propfgC(p)good} to  a finite group $L$ with order divisible
by $p$, which satisfies one of the five
assumptions  of Proposition \ref{propfgC(p)good} ($L$ is non-$p$-perfect;  it has a normal
Sylow $p$-subgroup, etc.), we obtain that Question \ref{qutwist} has a positive answer.

As application of Theorem \ref{thm12} we obtain classes of block algebras of finite groups
for which the answer of Question \ref{question} is positive. An inertial block is a block which is basic Morita equivalent
to its Brauer correspondent. We will discuss more details about inertial blocks and blocks of $p$-solvable groups in Section \ref{sec4}. Let $(P,e)$ be a fixed, maximal $b$-Brauer pair, where $e$ is a block of $kC_G(P)$.  Recall that the set of $b$-Brauer pairs is a $G$-poset. We denote by "$\leq$" the partial order relation on the set of $b$-Brauer pairs; see \cite{AlBr,AKO} for the theory of $b$-Brauer pairs and their generalizations.  We denote by $\mathcal{F}:=\mathcal{F}_{(P,e)}(G,b)$ the saturated fusion system associated with $b$ determined by the choice of $(P,e)$. This is the finite  category with objects all subgroups of $P$ and morphisms given by conjugation between $b$-Brauer pairs, see \cite{AKO,LiBo1} for more details regarding fusion systems of blocks. If $(Q,f)$ is a Brauer $b$-subpair of $(P,e)$ (i.e. $(Q,f)\leq (P,e)$) then $Q$ is called $\mathcal{F}$-centric if $Z(Q)$ is a defect group of $kC_G(Q)f$, see \cite[Proposition 8.5.3]{LiBo2}.
\begin{cor}\label{cor13} Let $B=kGb$ be a block algebra with nontrivial defect group $P$. Let $Q$ be an $\mathcal{F}$-centric subgroup of $P$ and let $f$ be the block of $kC_G(Q)$ such that $(Q,f)\leq (P,e)$. Assume
that one of the following statements is true:
\begin{itemize}
 \item[(a)] $P$ is normal in $G$;
 \item[(b)] $b$ is inertial;
 \item[(c)] $G$ is $p$-solvable;
 \item[(d)] $G=N_G(Q,f)$ and, one of the following conditions is satisfied:
 \subitem (i) $G/QC_G(Q)$ is non-$p$-perfect;
 \subitem (ii) $G/QC_G(Q)$ has a normal Sylow $p$-subgroup;
 \subitem(iii) $Z(P)$ is not included in $P'$;
 \subitem(iv) the exponent of $P'$ is strictly smaller than the exponent of $P$;
 \subitem(v) $P$ is metacyclic.
\end{itemize}

 Then $\HH^1(B)\neq 0.$
\end{cor}

It is well known that a nilpotent block is basic Morita equivalent to its defect group algebra and has nilpotent Brauer corespondent.  It follows that nilpotent blocks are inertial (see \cite[1.5]{PuZh} ), hence for nilpotent blocks Question \ref{question} has positive answer. The proof of Corollary \ref{cor13} is a consequence Proposition \ref{prop-iner-solv}.

Recall that a finite group $L$ with a Sylow $p$-subgroup $D$ is called $p$-perfect if $L=
\mathcal{O}^p(L)$. It is well known that  this definition is equivalent to $\Hc^1(L,\mathbb{F}_p)=0$ which, by 
some results reminded in Section \ref{sec2-reminder},
is the same as $\Hc^1(L,k)=0$. For shortness, sometimes, we will say that a  group $L$  is non-$p$-perfect if $\Hc^1(L,k)\neq 0$,
equivalently $\mathcal{O}^p(L)<L$. For any  $p$-group $Q$ of the finite group $G$, the Scott module $Sc(G,Q)$ is the unique (up to isomorphism) indecomposable
$kG$-module with vertex $Q$, with trivial source and having a submodule isomorphic to $k$
as trivial $kG$-module. In the second main result of this paper we present other blocks for which Question \ref{question} has 
positive answer. 


\begin{thm}\label{thm14} Let $B=kGb$ be a block algebra with nontrivial defect group $P$.
If $G$ is non-$p$-perfect and the Scott module $Sc(G,P)$ is isomorphic with the trivial $kG$-module $k$ then $\HH^1(B)\neq 0$.

 \end{thm}

 In Section \ref{sec3} we prove Theorem \ref{thm12} and Theorem \ref{thm14}. Section \ref{sec4} is devoted
 to the investigation of blocks which satisfies the assumptions of Theorem \ref{thm12} and \ref{thm14}. We end 
 this paper with examples of blocks which satisfy
 the assumptions of Theorem \ref{thm14}.
 
 We need to emphasize that although the proofs of the main results are not difficult,  we were able to   identify a sufficient
and clearly useful criterion for the nonvanishing of the first Hochschild cohomology  of twisted
group algebras in terms of properties of the twisting
$2$-cocycle. Also,  there are only few published results
on showing that the first Hochschild cohomology of a block with a nontrivial defect group is
nonzero.
 
\section{Reminder of two cocycles, the first group cohomology and Hochschild cohomology}\label{sec2-reminder}
In this section, if otherwise is not specified, $k$ is any field and $M$ is a trivial
$kL$-module, where $L$ is a finite group. It is well known that
$$\Hc^*(L,k)\otimes_k M\cong \Hc^*(L,M).$$
This isomorphism can be easily described using the Universal Coefficient Theorem and,
in particular we obtain
\begin{equation}\label{eq21}\Hc^1(L,M)\cong\Hc^1(L,k)\otimes_k M.
\end{equation}
Since we work with trivial $kL$-modules we have explicit 
identifications of vector spaces,
see \cite[Theorem 3.4.1]{Ca}
\begin{equation}\label{eq22}
\Hc^1(L,M)\cong \Hom_{Grp}((L,\cdot), (M,+)), \quad \Hc^1(L,k)\cong \Hom_{Grp}((L,\cdot), (k,+)).
\end{equation}
In the  case of the extension of fields $\mathbb{F}_p\leq k$, with $k$ 
a field
of characteristic $p$, by taking $\mathbb{F}_p$ the trivial $\mathbb{F}_pL$-module
and $k$ the trivial $kL$-module we get
\begin{equation}\label{eq23}
 \Hc^1(L,k)\cong\Hc^1(L,\mathbb{F}_p)\otimes_{\mathbb{F}_p}k
\end{equation}
The compatibility of cohomology and flat scalar extensions (of which
the following  proposition  is a special case) is well known. We leave the proof for the reader and just mention that, for the next
 statement (ii), one approach is to use statement (i) and centralizers decomposition of Hochschild cohomology of group algebras
 $$\HH^1(kL)\cong \Hc^1(L,kL)\cong \bigoplus_{g\in X}\Hc^1(C_L(g),k),$$
 where $X$ is a system of representatives of the conjugacy classes in $L$.
\begin{prop}\label{prop21H1}
 Let  $M$ be a trivial $kL$-module.
 \begin{itemize}
  \item[(i)] $\Hc^1(L,M)\neq 0$ if and only if $\Hc^1(L,k)\neq 0$. Particularly,
  from (\ref{eq23}), we obtain \\ $\Hc^1(L,k)\neq 0$ if and only if $\Hc^1(L,\mathbb{F}_p)\neq 0.$
  \item[(ii)] $\HH^1(kL)\neq 0$ if and only if $\HH^1(\mathbb{F}_p L)\neq 0$.
 \end{itemize}

\end{prop}

 A $2$-cocycle $\alpha\in Z^2(L,k^{\times})$, where $L$ acts trivially on $k^{\times}$, is a map $\alpha:G\times G\rightarrow k^{\times}$
satisfying
 \begin{equation}\label{eq25}
  \alpha(xy,z)\alpha (x,y)=\alpha(x,yz)\alpha(y,z),\forall x,y,z\in L.
 \end{equation}

 It is well known that 
 \begin{equation}\label{eq26}
 \alpha(1,x)=\alpha(x,1)=\alpha(1,1), \forall x\in L.
 \end{equation}

The following lemma shows that $2$-cocycles behave well with respect to cyclic
groups. 
\begin{lem}\label{lemcyclegrp}
 Let $\alpha\in Z^2(L, k^{\times})$, $x\in L$ and let $m,n$ be any integers.
Then $\alpha(x^m,x^n)=\alpha(x^n,x^m).$ 
\end{lem}
The above result is  well known by experts. One justification is based on the fact that restriction of
a $2$-cocycle to  the cyclic group $<x>$ becomes a $2$-coboundary in
$k$ (which is algebraically closed), and $2$-coboundaries of any abelian 
group are clearly symmetric in the two arguments. 

One of the main goals in this paper is the searching for finite 
groups $L$ which have an $\alpha$-regular element $x\in L$ such that $\Hc^1(C_L(x),k)
\neq 0$. For this we collect and recall some basic facts about $\alpha$-regular elements
and about $L_{C(p)}$. The following lemma is based on arguments which we find in \cite[Chapter 3, Lemma 6.1]{Ka}
and \cite[Lemma 1.2]{FJL}.
\begin{lem}\label{lem24}
  Let $L$ be a finite group with order divisible by $p$ and $\alpha\in Z^2(L,k^{\times})$. 
\begin{itemize}
 \item[(i)]  $1\in L_{\alpha}^{\circ}$;
 \item[(ii)] If $L$ is abelian then $L_{C(p)}=L$;
 \item[(iii)]If $L$ is cyclic then $L_{C(p)}=L=L_{\alpha}^{\circ}$;
 \item[(iv)] If $L$ is a $p$-group then $L_{C(p)}=L=L_{\alpha}^{\circ}$;
 \item[(v)] $L$ is non-$p$-perfect if and only if $1\in L_{C(p)}$;
 \item[(vi)] $L_{S(p)}\subseteq L_{C(p)}\cap L_{\alpha}^{\circ}$; 
 \item[(vii)] If $P$ is a normal Sylow $p$-subgroup of $L$ then
 $$\emptyset\neq P\setminus P'\subset L_{S(p)}.$$
\end{itemize}
\end{lem}
\begin{proof}
 \begin{itemize}
  \item[(i)] This statement is evident by (\ref{eq26}).
  \item[(ii)] If $L$ is abelian then for any $x\in L$ we have
  $$C_L(x)/(C_L(x))'=L/L'=L$$
  which has order divisible by $p$.
  \item[(iii)] The first equality is true by (ii) and the second equality follows from
  Lemma \ref{lemcyclegrp}.
  \item[(iv)] The first equality is evident while the second equality follows 
  from \cite[Chapter 3, Lemma 6.1 (iv)]{Ka}.
  \item[(v)] The identity element is in  $L_{C(p)}$ if and only if $p$ divides the index $|L:L'|$
   which, by \cite[Lemma 1.2 (3) (a)]{Ka}, is true if and only if $$\mathrm{foc}_L(S_p(1))<S_p(1),$$
   where $S_p(1)$ is a Sylow $p$-subgroup in $L$. But this last statement is true if and only if the 
   hyperfocal subgroup of $L$ is a proper subgroup of $S_p(1)$. This is the same to 
   $$\mathcal{O}^p(L)\cap S_p(1)<S_p(1)$$ which is true if and only if $\mathcal{O}^p(L)<L$. 
 \item[(vi)] The proof of this statement is clear by \cite[Lemma 1.2 (4)]{FJL} and 
 \cite[Chapter 3, Lemma 6.1]{Ka}.
 \item[(vii)] Since $P$ is a $p$-group, it is well known 
(by induction) that $P'<P$, so there is $x\in P\setminus P'$, which is 
obviously a $p$-element. But $P$ is a normal Sylow $p$-subgroup in $L$, hence 
$C_P(x)$ remains a normal Sylow $p$-subgroup of $C_{L}(x)$. 
Then, our $p$-element $x$ satisfies the property $$x\in C_P(x)\setminus (C_P(x))'.$$
By  \cite[Lemma 1.2 (2)]{FJL} it follows that $L$ satisfies the strong Non-Schur 
property $S(p)$ and, in fact,  we obtain
$$P\setminus P'\subset L_{S(p)}.$$ 
 \end{itemize}

\end{proof}

\begin{prop}\label{propfgC(p)good}
 Let $L$ be a finite group with order divisible by $p$ and $\alpha\in Z^2(L,k^{\times})$. 
Let $P$ be a Sylow $p$-subgroup in $L$. Assume that one of the following statements is true:
 \begin{enumerate}
     \item[(i)] $L$ is non-$p$-perfect;
     \item[(ii)] $P$ is normal in $L$;
     \item[(iii)] $Z(P)$ is not included in $P'$;
     \item[(iv)] the exponent of $P'$ is strictly smaller than the exponent of $P$;
     \item[(v)] $P$ is metacyclic.
 \end{enumerate}
 Then $L_{C(p)}\cap L_{\alpha}^{\circ}\neq \emptyset$.
\end{prop}
\begin{proof}
We approach the first two cases separately.

\textit{Case (i).} The result follows from Lemma \ref{lem24} (i) and (v) .

\textit{Case (ii).} In this case we apply Lemma \ref{lem24} (vii) and (vi).

\textit{Cases (iii), (iv) and (v).} Statement (2) of \cite[Lemma 1.2]{FJL} says that if $L$ satisfies (iii), or (iv), or (v) then $L$ satisfies the strong Non-Schur property $S(p)$, hence $L_{S(p)}\neq \emptyset$. The conclusion is given now by Lemma \ref{lem24} (vi).

\end{proof}
\begin{rem} Cyclic groups and $p$-groups (as in Lemma \ref{lem24} (iii) and (iv)) can also be included in the class of
finite groups satisfying one of the conditions of Proposition \ref{propfgC(p)good}, but in these cases the twisted group algebras are just ordinary group algebras.
\end{rem}
\section{Proofs of Theorem \ref{thm12} and \ref{thm14}} \label{sec3}
\begin{proof}\textbf{(of Theorem \ref{thm12}.)}
 The main ingredient is \cite[Lemma 3.5]{FJL} which, for our twisted group algebra, gives
\begin{equation}\label{eq1}
\HH^1(k_\alpha L)\cong \bigoplus\limits_{i=1}^t \Ext_{C_L(x_i)}^1(k,k\overline{x_i})
\end{equation}

where:
\begin{itemize}
    \item $\{x_i\}_{i\in\{1,...,t\}}$ is a set of representatives of the conjugacy classes in $L$;
    \item $k$ is the trivial $kC_L(x_i)$-module for any $i\in\{1,\ldots, t\}$;
    \item where $k\overline{x_i}$ denotes $k\otimes x_i$ and is given as $kC_L(x_i)$-module by 
    $$g(a\otimes x_i)=\alpha(g,x_i)(\alpha(x_i,g))^{-1}a\otimes x_i, \quad\forall a\in k,\forall g\in C_L(x_i).$$
\end{itemize}
Let $x\in L_{C(p)}\cap L_{\alpha}^{\circ}$.
Since $x\in L$ is $\alpha$-regular, it follows by \cite[Chapter 3, Lemma 6.1 (iii)]{Ka}
that there is $i_0\in\{1,...,t\}$ such that $x_{i_0}$ is $\alpha$-regular and $x$
is conjugate to $x_{i_0}$, hence
$$\alpha(g,x_{i_0})\ (\alpha(x_{i_0},g))^{-1}=1,\quad \forall g\in C_L(x_{i_0}).$$
The above statement means that  $k\overline{x_{i_0}}$ is isomorphic to $ k$ as trivial $kC_L(x_{i_0})$-module.
Next, from (\ref{eq1}) we obtain
\begin{equation}\label{eq2}\HH^1(k_\alpha L)\cong \Hc^1(C_L(x_{i_0}),k)\bigoplus\left(\bigoplus\limits_{\substack{i=1\\i\neq i_0}}^t\Hc^1(C_L(x_i),k\overline{x_i})\right)
\end{equation}
It is an easy exercise to verify that 
$$|C_L(x):(C_L(x))'|=|C_L(x_{i_0}):(C_L(x_{i_0}))'|,$$
since $x,x_{i_0}$ are conjugate. The element $x\in L$  gives  the Commutator index property $C(p)$ for $L $ thus $x_{i_0}\in L_{C(p)}$. This implies that the abelian group $C_L(x_{i_0})/(C_L(x_{i_0}))'$ has a quotient isomorphic to $C_p$, thus $$\Hc^1(C_L(x_{i_0})/C'_L(x_{i_0}),\mathbb{F}_p)\neq 0.$$
Consequently, by the inflation-restriction exact sequence we obtain $\Hc^1(C_L(x_{i_0}),\mathbb{F}_p)\neq 0,$ which by Proposition \ref{prop21H1} (i)
is equivalent to $$\Hc^1(C_L(x_{i_0}),k)\neq 0.$$
This last statement and (\ref{eq2}) assure us the conclusion.
\end{proof}

\begin{proof}\textbf{(of Theorem \ref{thm14})}

By \cite[Lemma 7]{KeLi} we know that $B$, as left $kG$-module by conjugation, decomposes 
$$B\cong Sc(G,P)\bigoplus \left(\bigoplus_{i=1}^r M_i\right),$$
where $M_i,i\in\{1,...,r\}$ are indecomposable $kG$-modules and $Sc(G,P)$ is the Scott $kG$-module with vertex $P$.
It is well known that $$\HH^1(B)\cong \Hc^1(G,B)$$
where $G$ acts by conjugation on $B$.
We obtain the decomposition
\begin{equation}\label{eq3}\HH^1(B)\cong \Hc^1(G,Sc(G,P))\bigoplus\left(\bigoplus_{i=1}^r \Hc^1(G,M_i)\right)
\end{equation}
Since $Sc(G,P)$ is the trivial $kG$-module $k$ we obtain
$$\Hc^1(G,Sc(G,P))\cong \Hc^1(G,k).$$
But $\Od^p(G)<G$ hence $\Hc^1(G,\mathbb{F}_p)\neq 0$ which,  by Proposition 
\ref{prop21H1} (i), gives $$\Hc^1(G,k)\neq 0.$$
Using $\Hc^1(G,Sc(G,P))\neq 0$ applied in (\ref{eq3}) we obtain the conclusion.
\end{proof}
\section{Inertial blocks, blocks of $p$-solvable groups and further remarks.}\label{sec4}
We begin with some properties of inertial blocks and blocks of $p$-solvable groups.
\begin{nim}\label{41}\textbf{Inertial blocks.}
Inertial blocks were introduced by Puig \cite{PuZh}. A block $b$ 
is inertial if it is basic Morita equivalent to $e$ as a block of $kN_G(P,e)$. 
Since $e$ has the same defect group $P\unlhd N_G(P,e)$, we obtain that an inertial
block $b$ is Morita equivalent to the block $e$ which has normal defect group $P$ in $N_{G}(P,e)$. In this case $b$ and $e$ have the same fusion system.
\end{nim}

\begin{nim}\label{42}\textbf{Blocks of $p$-solvable groups.} By results in \cite{Ku2}, 
see also \cite[Theorem 10.6.1]{LiBo2}, it is well known that if $b$ is a block of 
$kG$, with $G$ a $p$-solvable finite group, then there is a finite $p$-solvable 
group $L$ such that $B=kGb$ is Morita equivalent to $k_\alpha L$, where 
$[\alpha]\in H^2(L,k^{\times})$. Moreover $P$ is a Sylow $p$-subgroup of $L$, $\mathcal{O}_{p'}(L)=1$ and if $Q=\mathcal{O}_p(L)$
then $C_L(Q)=Q.$
\end{nim}

\begin{prop}\label{prop-iner-solv}
 Let $B=kGb$ be a block algebra with nontrivial defect group $P$. Let $Q$ be an $\mathcal{F}$-centric subgroup of $P$ and let $f$ be the block of $kC_G(Q)$ such that $(Q,f)\leq (P,e)$. Assume
that one of the following statements is true:
\begin{itemize}
 \item[(a)] $P$ is normal in $G$;
 \item[(b)] $b$ is inertial;
 \item[(c)] $G$ is $p$-solvable;
 \item[(d)] $G=N_G(Q,f)$ and, one of the following conditions is satisfied:
 \subitem (i) $G/QC_G(Q)$ is non-$p$-perfect;
 \subitem (ii) $G/QC_G(Q)$ has a normal Sylow $p$-subgroup;
 \subitem(iii) $Z(P)$ is not included in $P'$;
 \subitem(iv) the exponent of $P'$ is strictly smaller than the exponent of $P$;
 \subitem(v) $P$ is metacyclic.
\end{itemize}
 Then $B$ is Morita equivalent to a twisted group algebra $k_{\alpha}L$ 
 such that $L_{C(p)}\cap L_{\alpha}^{\circ}\neq\emptyset$.
\end{prop}
\begin{proof}
\begin{itemize}
 \item[(a)]  We denote by 
$$E=\Out_{\mathcal{F}}(P)=\Aut_{\mathcal{F}}(P)/\Inn(P)\cong N_G(P,e)/PC_G(P)$$
the inertial quotient of $b$, where $\mathcal{F}$ is the saturated fusion system of the block $b$. By \cite[Theorem 6.14.1]{LiBo2} (which goes back to K\"ulshammer \cite[Theorem A]{Ku1})
our  block algebra $kGb$ with normal defect group $P$ in $G$
 is Morita equivalent to a twisted 
group algebra $k_\alpha(P\rtimes E)$, where $\alpha\in Z^2(P\rtimes E,k^{\times})$ such
that 
$[\alpha]=\operatorname{inf}_E^{P\rtimes E}([\alpha])$ (by abuse of notation).
As a matter of fact $[\alpha]\in \mathcal{O}_{p'}(H^2(E,k^{\times}))$.  
But since $P \rtimes E$ has $P$ as normal Sylow $p$-subgroup we apply Proposition \ref{propfgC(p)good} (ii)
to obtain the conclusion.
\item[(b)] This statement is a consequence of (a) and \ref{41}.
\item[(c)]By \ref{42} we know that $B$ is Morita equivalent with $k_\alpha L$, 
where $\alpha\in Z^2(L,k^{\times})$ and $L$ is a $p$-solvable finite group such 
that $P$ remains a Sylow $p$-subgroup in $L$. 

We denote by $\mathcal{L}$ the class of all finite $p$-solvable groups with their
orders divisible by $p$. By $S_p(g)$ we understand a Sylow $p$-subgroup of $C_{L_0}(g)$, where $g\in L_0$ and $L_0\in\mathcal{L}$.
We adapt (almost verbatim) the methods from 
\cite{FJL} to show that any finite $p$-solvable group satisfies the strong 
Non-Schur property $S(p)$. It would then follow that $L$ has a $p$-element 
which gives the Commutator index property $C(p)$. 
This means that $L$ satisfies the assumption of Theorem \ref{thm12}, see Lemma \ref{lem24} (vi).

We argue by contradiction, so let $L_0\in\mathcal{L}$ be a minimal 
counterexample to the strong Non-Schur property $S(p)$. Obviously $L_0$ is non-abelian
and we show next that $L_0$ is a simple finite $p$-solvable group. 
By contradiction assume that there is $N\trianglelefteq L_0$ such 
that $\{1\}\neq N\neq L_0$. We split the proof in 
in two cases:
\begin{enumerate}
    \item[(1)]  $p$ divides $|L_0/N|$. It is well known that $L_0/N$ 
   remains a finite $p$-solvable group, thus $L_0/N\in\mathcal{L}$. Consequently 
    $|L_0/N|<|L_0|$, therefore $L_0/N$ has the strong Non-Schur 
    property $S(p)$. So,  there is $g\in L$ such that $\overline{g}$ (which is  notation 
    for $gN$) is a $p$-element with $$\overline{g}\in S_p(\overline{g})\setminus 
    (S_p(\overline{g}))', g_p\in (S_p(g_p))',$$
    where $g_p$ is the $p$-part of $g$.
    Next, since $\overline{S_p(g_p)}\leq S_p(\overline{g_p})$ it 
    follows that
    $$\overline{g}=\overline{g_p}\in \overline{S_p(g_p)'}\leq \overline{S_p(g_p)}'\leq 
    (S_p(\overline{g_p}))' =(S_p(\overline{g}))',$$
    which is a contradiction.
    \item[(2)] $p$ is not dividing $|L_0/N|$. It is well known that $N$ 
  remains a finite $p$-solvable group such that $p$ divides $ |N|$. 
    Thus $N\in\mathcal{L}$, with $|N|< |L_0|$. Then $N$ satisfies the strong 
    Non-Schur property $S(p)$. It follows that there is $g\in N$ a $p$-element
    such that
    $$g\in S_p^N(g)\backslash (S_p^N(g))',$$
    where $S_p^N(g)$ is a Sylow $p$-subgrup of $C_N(g)$. 
    Since $L_0/N$ is a $p'$-group it is known that $S_p^N(g)=S_p(g)$. 
  Consequently the element $g\in N$ is a $p$-element such that $g\in S_p(g)\backslash(S_p(g))'$, 
    a contradiction.
    
    From (1), (2) we obtain that $L_0\in\mathcal{L}$ 
    is a finite simple $p$-solvable group, hence $L_0$ must be a $p$-group. 
    But the only finite simple $p$-group is the cyclic group $C_p$, which is abelian.
\end{enumerate}

\item[(d)] By \cite[Corollary 8.12.9]{LiBo2} we obtain that $B$ is Morita equivalent to a twisted group algebra $k_{\alpha}L$, where $L$ is a finite group having $P$ as a Sylow $p$-subgroup, such that $Q$ is normal in $L$ and $L/Q\cong G/QC_G(Q)$.
\begin{itemize}
\item[(i)] Since $G/QC_G(Q)\cong L/Q$ is non-$p$-perfect, it follows that $L$ is non-$p$-perfect (one argument is the inflation-restriction exact sequence). Next,  Proposition \ref{propfgC(p)good} (i) assure us the conclusion.
\item[(ii)] Since $G/QC_G(Q)\cong L/Q$ has a normal Sylow $p$-subgroup, it follows that $L$ has the same property and we apply Proposition \ref{propfgC(p)good} (ii).

\item[(iii),] (iv) and (v): 

For proving the conclusion under one of these assumptions we apply Proposition \ref{propfgC(p)good} (iii), (iv) and (v), keeping in mind that $P$ becomes a Sylow $p$-subgroup of $L$.
\end{itemize}

\end{itemize}

\end{proof}
Recall that the principal block of $kG$ is 
the unique block of $kG$ contained in the trivial $kG$-module $k$. 
Defect groups of principal blocks are Sylow $p$-subgroups.
If $b$ is the principal block of a non-$p$-perfect finite group $G$, since 
$Sc(G,P)\cong k$ as trivial $kG$-module, we are under the assumptions of              
Theorem \ref{thm14}. 
There exist also non-principal blocks for which the defect groups are the Sylow $p$-subgroups, see \cite[Remark 1]{KeKK}
and \cite[Example 10.2.1]{LiBo2}.

In the following remark, statement a), we mention an example of a block which is not principal
but has defect group a Sylow $p$-subgroup and, for which the underlying group is 
non-$p$-perfect; thus, we are still under the assumptions of Theorem \ref{thm14}. 
\begin{rem}\begin{itemize}
\item[a)] Suppose $p=3$ and set $G=SL_2(3)$. The order of $G$ is $2^3\cdot 3$ and the group
$$P=\left\{\left(\begin{array}{cc}
     1&  b\\
     0& 1
\end{array}\right)| b\in \mathbb{F}_3\right\}$$
is a Sylow $3$-subgroup of $G$.
The block  $b_1$ of $kG$  from  \cite[Example 6.7.9]{LiBo2} is a non-principal block with 
defect group $P$. Since
\begin{align*}
      & H=\left\{ \left(\begin{array}{cc}
     1&  0\\
     0& 1
\end{array}\right),\left(\begin{array}{cc}
     2&  0\\
     0& 2
\end{array}\right),\left(\begin{array}{cc}
     0&  2\\
     1& 0
\end{array}\right),\left(\begin{array}{cc}
     0&  1\\
     2& 0
\end{array}\right),\left(\begin{array}{cc}
     1&  1\\
     1& 2
\end{array}\right), \right. \\
      & \hspace{1cm} \left.\left(\begin{array}{cc}
     2&  2\\
     2& 1
\end{array}\right),  \left(\begin{array}{cc}
     2&  1\\
     1& 1
\end{array}\right),\left(\begin{array}{cc}
     1&  2\\
     2& 2
\end{array}\right) \right\}
\end{align*}
 is the only normal $2$-subgroup (isomorphic to $Q_8$) and $|G:H|=3$, it follows that $\mathcal{O}^3(G)<G$.
 \item[b)]For any saturated fusion system $\mathcal{F}$ on an arbitrary $p$-group
$P$ we know that the first cohomology of the fusion system is
$$\Hc^1(P,k)^{\mathcal{F}}=\Hom_{Grp}(P/P^p\mathrm{foc}(\mathcal{F}),\cdot),(k,+)).$$
Consequently  it is easy to verify that $\Hc^1(P,k)^{\mathcal{F}}\neq 0$ if and only if $\mathrm{foc}(\mathcal{F})<P$. If we apply this to finite group block algebras $\mathcal{F}=\mathcal{F}_{(P,e)}(G,b)$, in conjunction with the embedding \cite[Theorem 5.6]{Li2} of the first block cohomology $\Hc^1(P,k)^{\mathcal{F}}$ into the first Hochschild cohomology $\HH^1(B)$, we obtain a sufficient criterion for a positive answer of Question \ref{question}, see \cite[Proposition 12.8]{Li}. One approach for the future study of Question \ref{question} is to investigate  block algebras $B$ with $\mathrm{foc}(\mathcal{F})=P$.
\end{itemize}
\end{rem}





\end{document}